\documentclass[reqno]{amsart}

\usepackage[T1]{fontenc}

\usepackage{amsmath,amsfonts,amssymb,amsthm,amscd}
\usepackage{latexsym}
\usepackage{cite}
\usepackage{mathrsfs}
\usepackage{color}
\usepackage{comment}
\usepackage{float}
\usepackage{hyperref}
\usepackage{bbm}
\hypersetup{
  colorlinks=true,
  linkcolor=blue,
  citecolor=blue,
  urlcolor=blue
}

\newtheorem{thm}{Theorem}[section]
\newtheorem{theorem}[thm]{Theorem}
\newtheorem{proposition}[thm]{Proposition}

\newtheorem{lemma}[thm]{Lemma}

\newtheorem{corollary}[thm]{Corollary}

\theoremstyle{definition}

\newtheorem{example}[thm]{Example}
\newtheorem{question}[thm]{Question}

\theoremstyle{remark}
\newtheorem{remark}[thm]{Remark}

\newcommand{\cA}{\mathcal A}

\newcommand{\cM}{\mathcal M}

\newcommand{\cQ}{\mathcal Q}

\newcommand{\C}{\mathbb C}
\newcommand{\Z}{\mathbb Z}
\newcommand{\one}{\mathbf 1}

\newcommand{\Span}{\operatorname{span}_{\C}}
\newcommand{\Ann}{\operatorname{Ann}}
\newcommand{\End}{\operatorname{End}}
\newcommand{\Hom}{\operatorname{Hom}}
\newcommand{\Inder}{\operatorname{Inder}}
\newcommand{\Der}{\operatorname{Der}}
\newcommand{\id}{\operatorname{id}}

\numberwithin{equation}{section}
\allowdisplaybreaks
\begin{document}

%================================================
% Title and author information
%================================================

\title[The Pozhidaev and Cantarini--Kac Constructions]
{The Pozhidaev and Cantarini--Kac Constructions of Simple \(n\)-Lie Algebras: Distinctions and Realizations}
\author[Xinru Cao]{Xinru Cao}

\address[Xinru Cao]{%
School of Mathematics and Statistics,
Northeast Normal University,
Changchun 130024, China}

\email{caoxinru@nenu.edu.cn}

\author[Bakhrom A. Omirov]{Bakhrom A. Omirov}

\address[Bakhrom A. Omirov]{%
Institute for Advanced Study in Mathematics,
Harbin Institute of Technology,
Harbin 150001, China}

\address[Bakhrom A. Omirov]{%
Suzhou Research Institute,
Harbin Institute of Technology,
Suzhou 215104, China}

\email{omirovb@mail.ru}

\author[Yuhui Tan]{Yuhui Tan}
\address[Yuhui Tan]{Institute for Advanced Study in Mathematics,
Harbin Institute of Technology,
Harbin 150001, China}

\email{26B312001@stu.hit.edu.cn}

%================================================
% Classification and keywords
%================================================

% Compatibility with older installations of amsart.
\makeatletter
\@ifundefined{subjclassname@2020}{%
  \@namedef{subjclassname@2020}{\textup{2020} Mathematics Subject Classification}%
}{}
\makeatother
\subjclass[2020]{17A42, 17B65}

\keywords{$n$-Lie algebra; Filippov algebra; inner derivation;
divergence-free Lie algebra; Jacobian algebra}

%================================================
% Abstract
%================================================

\begin{abstract}
	Let $n\geq3$, let $H\subseteq\C^n$ be any additive subgroup
	spanning $\C^n$, and let $0\ne t\in H$.
	We study Pozhidaev's central simple $n$-Lie algebra
	$P(H,t)=\widetilde{\mathcal A}(H,t)/\C e_0$ without a finite
	generation or discreteness assumption on $H$.
	Its inner derivation algebra is the simple generalized
	divergence-free Lie algebra $\mathcal S(0,0,n;t,H)$.
	We prove that its space of inner-equivariant symmetric products
	vanishes and that every $1/n$-derivation is a scalar multiple of
	the identity. Using these invariants, we show that $P(H,t)$
	is not isomorphic to any simple nonabelian $n$-Lie algebra
	defined on the underlying spaces of the $S$, $W$, or $SW$
	constructions recorded by Cantarini and Kac.
	We also show that $P(H,t)$ is the quotient by the constants
	of the derived algebra of an explicit $S$-algebra on $\C[H]$.
	Finally, we realize Pozhidaev's second construction $E(H)$
	over $\C$ as a $W$-algebra.
\end{abstract}

\maketitle

\section{Introduction}

An $n$-Lie algebra, also called a Filippov algebra, is a vector space equipped with an alternating $n$-linear product whose adjoint operators are derivations. These algebras originated in Nambu's higher Hamiltonian formalism \cite{Nambu} and were axiomatized by Filippov \cite{Filippov}. Over an algebraically closed field of characteristic zero, the finite-dimensional simple $n$-Lie algebra is unique up to isomorphism for every $n\geq3$ and has dimension $n+1$ \cite{Ling}. Infinite-dimensional examples exhibit considerably more variety. Determinant, Jacobian, and Wronskian constructions provide important classes of such algebras \cite{DzhumadildaevJacobian,DzhumadildaevWronskians,Takhtajan,JumaniyozovOmirov}.

Cantarini and Kac classified simple linearly compact $n$-Lie superalgebras in characteristic zero \cite{CantariniKac}. Their Appendix~A also describes three algebraic constructions,
$$S(A,\mathfrak g),\qquad W(A,\mathfrak g),\qquad SW(A,D).$$
Pozhidaev had earlier constructed central simple $n$-Lie algebras from additive subgroups of finite-dimensional vector spaces \cite{PozhidaevMonomial,PozhidaevCentral}.
The present paper studies intrinsic operator spaces of his first construction and uses them to compare these algebras with the three full-space constructions above. We also describe explicitly the central reduction that relates the first construction to an $S$-algebra, and realize his second construction $E(H)$ as a $W$-algebra.

Throughout the paper the ground field is $\C$, and $n\geq3$. Let $H\subseteq\C^n$ be an additive subgroup satisfying $\Span H=\C^n$, and let $0\ne t\in H$. Write
$$P(H,t)=\widetilde{\mathcal A}(H,t)/\C e_0.$$
This algebra has a basis $\{u_a:a\in H\setminus\{0,t\}\}$ with bracket
\begin{equation}\label{eq1.1}
[u_{a_1},\ldots,u_{a_n}]=\det(a_1,\ldots,a_n)u_{a_1+\cdots+a_n+t},\qquad u_0=u_t=0.
\end{equation}
Our notation includes the lattice specialization $P_{n,t}=P(\Z^n,t)$.

We prove that
\[
\Inder P(H,t)\cong\mathcal S(0,0,n;t,H),
\]
where the right-hand side is a simple generalized divergence-free Lie
algebra studied by Su and Xu \cite{SuXu}. We also determine the
associated operator spaces $\cM(P(H,t))$ and $\cQ(P(H,t))$, obtaining
\[\cM(P(H,t))=0, \qquad \cQ(P(H,t))=\C\id_{P(H,t)}.
\]
Associative multiplication supplies a nonzero element of $\cM(S(A,\mathfrak g))$ whenever the bracket is nonzero. For the $W$-construction, left multiplication embeds $A$ into $\cQ(W(A,\mathfrak g))$ under the hypotheses stated below. The inner derivation algebra of a nonabelian $SW(A,D)$ with $D$-simple $A$ has a nonzero proper ideal of $A$-linear operators. These facts yield the following comparison.

\begin{theorem}[Main theorem]\label{thm1.1}
Let $n\geq3$, let $H\subseteq\C^n$ be an additive subgroup with $\Span H=\C^n$, and let $0\ne t\in H$. The central simple $n$-Lie algebra $P(H,t)$ is not isomorphic to any simple nonabelian algebra obtained on the full underlying space by one of the constructions
$$S(A,\mathfrak g),\qquad W(A,\mathfrak g),\qquad SW(A,D)$$
under the hypotheses of Section~\ref{subsec:CK-constructions}. Here the full underlying space is $A$ for $S$ and $W$, and $A^{\oplus(n-1)}$ for $SW$; no derived subalgebra or central quotient is included in this statement.
\end{theorem}

This qualification is essential. In Section~\ref{sec7} we construct an abelian $n$-dimensional Lie algebra $\mathfrak g_t$ of derivations of the group algebra $\C[H]$ such that
\begin{equation}\label{eq1.2}
P(H,t)\cong\frac{[S,\ldots,S]}{\langle\mathbbm{1}\rangle},\qquad S=S(\C[H],\mathfrak g_t).
\end{equation}
Thus the nonisomorphism theorem does not assert that $P(H,t)$ cannot be obtained from the three constructions by passing to derived subalgebras and central quotients. In particular, it does not settle the existence question in \cite[Appendix~A]{CantariniKac} under an interpretation allowing those operations.

Section~\ref{sec2} gives the definitions. Section~\ref{sec3} identifies the inner derivation algebra. Sections~\ref{sec4} and~\ref{sec5} determine $\cM(P(H,t))$ and $\cQ(P(H,t))$, respectively, and Section~\ref{sec6} completes the comparison with the three constructions. Section~\ref{sec7} proves \eqref{eq1.2}; the final section realizes $E(H)$ as a $W$-algebra.

\section{Preliminaries and comparison constructions}\label{sec2}

An \emph{$n$-Lie algebra} is a complex vector space $L$ equipped with an alternating $n$-linear bracket satisfying the fundamental identity
$$[x_1,\ldots,x_{n-1},[y_1,\ldots,y_n]]=\sum_{i=1}^n [y_1,\ldots,[x_1,\ldots,x_{n-1},y_i],\ldots,y_n].$$
For $x_1,\ldots,x_{n-1}\in L$, the map $\operatorname{ad}(x_1,\ldots,x_{n-1})(y) =[x_1,\ldots,x_{n-1},y]$ 
is therefore a derivation of $L$. The Lie subalgebra of $\End_{\C}(L)$ spanned by these maps is denoted by $\Inder(L)$.  An ideal of $L$ is a subspace $I$ such that $[I,L,\ldots,L]\subseteq I$.  We call $L$ simple when $[L,\ldots,L]\neq0$ and its only ideals are $0$ and $L$. The centroid of an $n$-Lie algebra $L$, denoted by
$\operatorname{Cent}(L)$, consists of all $\C$-linear maps $T:L\to L$ satisfying
$$T([x_1,\ldots,x_n])=[x_1,\ldots,T(x_i),\ldots,x_n]$$
for all $x_1,\ldots,x_n\in L$ and every $1\leq i\leq n$. A simple algebra $L$ is called \emph{central simple} if $\operatorname{Cent}(L)=\C\id_L$.

\subsection{Pozhidaev's algebra for an arbitrary spanning subgroup}

Fix an additive subgroup $H\subseteq\C^n$ with $\Span H=\C^n$ and an element $0\ne t\in H$. On
$$\mathcal A(H,t)=\bigoplus_{a\in H}\C e_a$$
consider the determinant bracket
\begin{equation}\label{eq2.1}
[e_{a_1},\ldots,e_{a_n}]=\det(a_1,\ldots,a_n)e_{a_1+\cdots+a_n+t}.
\end{equation}
By Pozhidaev's construction \cite[Theorem~2.1]{PozhidaevMonomial}, \eqref{eq2.1} defines an $n$-Lie algebra structure on $\mathcal A(H,t)$. An explicit realization as an $S$-algebra is also given in Section~\ref{sec7}. 

If $a_1+\cdots+a_n=0$, the determinant vanishes. Consequently
$$\widetilde{\mathcal A}(H,t)=\bigoplus_{a\in H\setminus\{t\}}\C e_a$$
is an $n$-Lie subalgebra. Since $t\ne0$, it contains the central element $e_0$. Define
$$P(H,t)=\widetilde{\mathcal A}(H,t)/\C e_0,\qquad \Lambda=H\setminus\{0,t\}.$$
The images $u_a=e_a+\C e_0$, $a\in\Lambda$, form a basis, and their bracket is \eqref{eq1.1}. We always interpret $u_0=u_t=0$.

Since $2t\in H\setminus\{t\}$, this set still spans $\C^n$. Thus $P(H,t)$ is central simple by \cite[Theorem~2.2]{PozhidaevCentral}. It is infinite-dimensional because $H$ contains the infinite cyclic subgroup $\Z t$.

For the weight calculations set
$$K=H\cap\C t.$$
This is an arbitrary additive subgroup of the line $\C t$. 

\subsection{The constructions in Appendix~A of Cantarini--Kac}\label{subsec:CK-constructions}

We use the determinant formulas recorded in \cite[Appendix~A]{CantariniKac}, with their underlying spaces made explicit. Let $A$ be a commutative associative complex algebra, not necessarily unital, and let $\mathfrak g\subseteq\Der A$ be a Lie algebra. Assume that $A$ has no nonzero proper associative ideal invariant under every element of $\mathfrak g$; we call this condition $\mathfrak g$-simplicity.

If $\dim\mathfrak g=n$ and $D_1,\ldots,D_n$ is a basis of $\mathfrak g$, then $S(A,\mathfrak g)$ is the vector space $A$ with bracket
\begin{equation}\label{eq2.2}
 [f_1,\ldots,f_n]_S
 =\det\!\begin{pmatrix}
 D_1(f_1)&\cdots&D_1(f_n)\\
 \vdots&&\vdots\\
 D_n(f_1)&\cdots&D_n(f_n)
 \end{pmatrix}.
\end{equation}

If $\dim\mathfrak g=n-1$ and $D_1,\ldots,D_{n-1}$ is a basis of $\mathfrak g$, then $W(A,\mathfrak g)$ is again the vector space $A$, now with bracket
\begin{equation}\label{eq2.3}
 [f_1,\ldots,f_n]_W =\det\!\begin{pmatrix}
 f_1&\cdots&f_n\\
 D_1(f_1)&\cdots&D_1(f_n)\\
 \vdots&&\vdots\\
 D_{n-1}(f_1)&\cdots&D_{n-1}(f_n)
 \end{pmatrix}.
\end{equation}

For the third construction, take $\mathfrak g=\C D$ and assume that $A$ has no nonzero proper $D$-invariant ideals.  Put
$$ SW(A,D)=A^{\langle1\rangle}\oplus\cdots\oplus A^{\langle n-1\rangle},$$
where $A^{\langle j\rangle}$ is a copy of $A$ and $f^{\langle j\rangle}$ denotes the copy of $f\in A$.  A bracket of homogeneous elements is zero unless every color $1,\ldots,n-1$ occurs.  If color $k$ is repeated, the defining nonzero bracket is
\begin{align}\label{eq2.4}
 &[f_1^{\langle1\rangle},\ldots, f_{k-1}^{\langle k-1\rangle},f_k^{\langle k\rangle}, f_{k+1}^{\langle k\rangle},f_{k+2}^{\langle k+1\rangle},\ldots, f_n^{\langle n-1\rangle}]\notag\\
 &\quad=(-1)^{k+n-1} \bigl(f_1\cdots f_{k-1} (D(f_k)f_{k+1}-f_kD(f_{k+1})) f_{k+2}\cdots f_n\bigr)^{\langle k\rangle},
\end{align}
extended by alternating multilinearity.

Throughout, these symbols denote the algebras on the underlying
vector spaces specified above. The $\mathfrak g$-simplicity of $A$
does not in general imply simplicity of the corresponding
$n$-Lie algebra; for example, $1$ is central in $S(A,\mathfrak g)$
whenever $A$ is unital. Our nonisomorphism results concern the
simple nonabelian members of these constructions. Derived
subalgebras and central quotients are treated separately in
Section~\ref{sec7}.

\section{The Lie algebra of inner derivations}\label{sec3}

We first determine the homogeneous components of $\Inder P(H,t)$. Its degree-zero subalgebra will provide the weight decomposition used in Section~\ref{sec4}.

For $r\in H$ and $\lambda\in(\C^n)^*$, define an operator on $P(H,t)$ by
$$ X_{r,\lambda}(u_a)=\lambda(a)u_{a+r},\qquad a\in\Lambda,$$
where $u_0=u_t=0$.

\begin{proposition}
We have
\begin{equation}\label{eq3.1}
 \Inder P(H,t)=\bigoplus_{\substack{r\in H\\r\ne t}}
 \{X_{r,\lambda}:\lambda(r-t)=0\}.
\end{equation}
In particular, $\mathfrak h=\{X_{0,\lambda}:\lambda(t)=0\}$ is an $(n-1)$-dimensional abelian subalgebra of $\Inder P(H,t)$.
\end{proposition}

\begin{proof}
An elementary inner derivation has the form
$$\operatorname{ad}(u_{p_1},\ldots,u_{p_{n-1}})(u_a)=\det(p_1,\ldots,p_{n-1},a)u_{a+p_1+\cdots+p_{n-1}+t}.$$
Put $r=p_1+\cdots+p_{n-1}+t$ and $\lambda(a)=\det(p_1,\ldots,p_{n-1},a)$. Then $\lambda(r-t)=0$. If $r=t$, the relation $\sum\limits_i p_i=0$ makes $p_1,\ldots,p_{n-1}$ linearly dependent, and hence $\lambda=0$. This proves one inclusion in \eqref{eq3.1}.

Conversely, fix $r\ne t$ and put $v=r-t$. Since $H$ spans $\C^n$, there exist $q_1,\ldots,q_{n-1}\in H$ such that $q_1,\ldots,q_{n-1},v$ form a basis. The functionals
$$\mu_j(a)=\det(q_1,\ldots,\widehat q_j,\ldots,q_{n-1},v,a),\qquad 1\le j\le n-1,$$
form a basis of $\Ann(v)$. For a fixed $j$, take $p_1,\ldots,p_{n-2}$ to be the vectors $m q_k$ with $k\ne j$, in their natural order, and put $p_{n-1}=v-p_1-\cdots-p_{n-2}.$
Here $m$ is a nonzero integer. The $q_k$ are nonzero and their indicated sum is nonzero. Consequently only finitely many choices of $m$ can make one of the $p_i$ equal to $0$ or $t$. Choose any other $m$. Then all $p_i\in\Lambda$, their sum is $v$, and
$$\det(p_1,\ldots,p_{n-1},a)=m^{n-2}\mu_j(a).$$
Thus every $X_{r,\mu_j}$ is inner, proving the reverse inclusion.

For each fixed $r$, the map $\lambda\mapsto X_{r,\lambda}$ is injective. Indeed, if $\lambda\ne0$, choose $a\in H$ with $\lambda(a)\ne0$. For all but finitely many $m\in\Z\setminus\{0\}$, both $ma$ and $ma+r$ belong to $\Lambda$, and hence
$$X_{r,\lambda}(u_{ma})=m\lambda(a)u_{ma+r}\ne0.$$
For each $a\in\Lambda$, the nonzero terms corresponding to distinct values of $r$ involve distinct basis vectors. Comparing coefficients for all $a\in\Lambda$ therefore proves that the sum in \eqref{eq3.1} is direct. Finally, the operators $X_{0,\lambda}$ are diagonal in the basis $\{u_a:a\in\Lambda\}$, so $\mathfrak h$ is abelian. The injectivity established above gives $\dim\mathfrak h=\dim\Ann(t)=n-1$.
\end{proof}

For the operators occurring in \eqref{eq3.1}, the commutator is
\begin{equation}\label{eq3.2}
 [X_{r,\lambda},X_{s,\mu}]
 =X_{r+s,\,\lambda(s)\mu-\mu(r)\lambda}.
\end{equation}
The deleted indices cause no extra terms. An intermediate index $t$ has coefficient $\mu(t-s)=0$ or $\lambda(t-r)=0$, while an intermediate index $0$ is annihilated by the next operator. If $\nu=\lambda(s)\mu-\mu(r)\lambda$, then $\nu(r+s-t)=0$. If $r+s=t$, both $\lambda(s)$ and $\mu(r)$ vanish, so the right-hand side is zero.

We first define the generalized divergence-free Lie algebra
that occurs below. Let
$$\mathcal A_H=\C[H]=\bigoplus_{a\in H}\C x^a,\qquad x^a x^b=x^{a+b}.$$
For $\lambda\in(\C^n)^*$, define $\partial_\lambda(x^a)=\lambda(a)x^a$, and put $T=\{\partial_\lambda:\lambda\in(\C^n)^*\}.$
The generalized Witt algebra $\mathcal W_H=\mathcal A_HT\subseteq\Der(\mathcal A_H)$
has bracket
\begin{equation}\label{eq3.3}
[x^r\partial_\lambda,x^s\partial_\mu]=x^{r+s}\partial_{\lambda(s)\mu-\mu(r)\lambda}.
\end{equation}

Write $\partial_i(x^a)=a_i x^a$. For $r\in H$ and $1\leq p<q\leq n$, set
$$D_{p,q}(x^r)=x^r\big((r_q-t_q)\partial_p-(r_p-t_p)\partial_q\big).$$
Following the generating-space convention of \cite[(2.4)--(2.5)]{SuXu}, define
$$\mathcal S(0,0,n;t,H)=\Span\{D_{p,q}(x^r):r\in H,\ 1\leq p<q\leq n\}.$$
The degree-$t$ generators vanish. For $r\ne t$, the coefficient functionals of the degree-$r$ generators span $\Ann(r-t)$. Consequently,
\begin{equation}\label{eq3.4}
\mathcal S(0,0,n;t,H)\cong \bigoplus_{\substack{r\in H\\r\ne t}}\{x^r\partial_\lambda:\lambda(r-t)=0\}.
\end{equation}
Its bracket is inherited from $\mathcal W_H$. Indeed, if $\lambda(r-t)=\mu(s-t)=0$, then $\lambda(s)\mu-\mu(r)\lambda$ annihilates $r+s-t$; if $r+s=t$, both $\lambda(s)$ and $\mu(r)$ vanish. Thus \eqref{eq3.3} makes \eqref{eq3.4} a Lie subalgebra of $\mathcal W_H$.

\begin{theorem}\label{thm3.2}
There is a Lie algebra isomorphism
$$\Inder P(H,t)\cong\mathcal S(0,0,n;t,H).$$
In particular, $\Inder P(H,t)$ is simple.
\end{theorem}

\begin{proof}
By \eqref{eq3.1} and \eqref{eq3.4}, the linear map defined by
$X_{r,\lambda}\longmapsto x^r\partial_\lambda$ is bijective. Equations \eqref{eq3.2} and \eqref{eq3.3} show that it preserves Lie brackets.

The pairing $\langle\partial_\lambda,a\rangle=\lambda(a)$ is nondegenerate, since $H$ spans $\C^n$. As $\dim T=n\geq3$, the algebra in \eqref{eq3.4} is simple by \cite[\S2.3, (2.12) and the following paragraph]{ZhaoSpecial}. The simplicity of $\Inder P(H,t)$ follows from the isomorphism above.
\end{proof}

\section{Equivariant symmetric products}\label{sec4}

For an $n$-Lie algebra $L$, put
$$\cM(L)=\Hom_{\Inder L}(S^2L,L).$$
Thus $m\in\cM(L)$ is a symmetric bilinear map satisfying
\begin{equation}\label{eq4.1}
 Dm(x,y)=m(Dx,y)+m(x,Dy)\qquad(D\in\Inder L).
\end{equation}

The notion of a $\delta$-derivation was introduced for Lie algebras by Filippov \cite{FilippovDelta}.

Similarly, for an $n$-Lie algebra $L$ and $\delta\in\C$, a linear map $T:L\to L$ is called $\delta$-derivation if
$$T([x_1,\ldots,x_n])=\delta\sum_{i=1}^n[x_1,\ldots,T(x_i),\ldots,x_n].$$
Set $\cQ(L)=\operatorname{Der}_{1/n}(L)$; equivalently,
\begin{equation}\label{eq4.2}
 \sum_{i=1}^n[x_1,\ldots,T(x_i),\ldots,x_n]=nT([x_1,\ldots,x_n]).
\end{equation}
Every centroid element satisfies this identity, so $\operatorname{Cent}(L)\subseteq\cQ(L)$. Theorem~\ref{thm5.3} will prove
$$\operatorname{Der}_{1/n}(P(H,t))=\C\id_{P(H,t)}=\operatorname{Cent}(P(H,t)).$$
This excludes $\frac1n$-derivations outside the centroid, in addition to the central simplicity established in \cite{PozhidaevCentral}. In contrast, Proposition~\ref{prop5.4} supplies an injective map $A\to\cQ(W(A,\mathfrak g))$ by associative multiplication.

\begin{lemma}\label{lem4.1}
If $\varphi:L\to L'$ is an $n$-Lie algebra isomorphism, then
$$\cQ(L)\longrightarrow\cQ(L'),\qquad T\longmapsto\varphi T\varphi^{-1},$$
and
$$\cM(L)\longrightarrow\cM(L'),\qquad m\longmapsto\bigl((x,y)\mapsto \varphi m(\varphi^{-1}x,\varphi^{-1}y)\bigr)$$
are linear isomorphisms. In particular, their dimensions are isomorphism invariants.
\end{lemma}

\begin{proof}
Conjugation identifies $\Inder L$ with $\Inder L'$. Substitution in \eqref{eq4.1} and \eqref{eq4.2} proves the claims, and transport by $\varphi^{-1}$ gives the inverse maps.
\end{proof}

We now use the $\mathfrak h$-weight decomposition of $P(H,t)$ to restrict the possible values of an equivariant symmetric product. Equivariance under the remaining inner derivations then forces these values to vanish.

\begin{theorem}\label{thm4.2}
We have $\cM(P(H,t))=0$.
\end{theorem}

\begin{proof}
Put $K=H\cap\C t$. For $v\in H$, define a weight on $\mathfrak h$ by $\chi_v(X_{0,\lambda})=\lambda(v)$. The vector $u_a$ has weight $\chi_a$, and the common annihilator of all functionals vanishing on $t$ is $\C t$. Therefore
$$\chi_a=\chi_b\quad\Longleftrightarrow\quad a-b\in K.$$
If $m\in\cM(P(H,t))$, equivariance under $\mathfrak h$ yields
\begin{equation}\label{eq4.3}
 m(u_a,u_b)=\sum_{s\in K}\alpha_s u_{a+b+s},
\end{equation}
where each such sum is finite and terms with index $0$ or $t$ are omitted. We show that all its coefficients vanish.

\noindent\emph{Step 1: $a,b,t$ are linearly independent.}
Choose $\eta\in(\C^n)^*$ satisfying $\eta(a+t)=\eta(b)=0$ and $\eta(t)\ne0$. Then $X_{-a,\eta}$ is inner and annihilates both inputs: it sends $u_a$ to the deleted index $0$ and has coefficient zero on $u_b$. On the other hand,
$$X_{-a,\eta}(u_{a+b+s})=\eta(s-t)u_{b+s}.$$
The output indices are valid and pairwise distinct because $b\notin\C t$. As $s-t\in\C t$, the coefficient $\eta(s-t)$ vanishes exactly when $s=t$. Equivariance therefore gives $m(u_a,u_b)=\alpha_tu_{a+b+t}$. Choose $\mu$ with $\mu(a)=\mu(b)=0$ and $\mu(t)\ne0$. The inner operator $X_{t+a,\mu}$ annihilates both inputs, whereas
$$X_{t+a,\mu}(u_{a+b+t})=\mu(t)u_{2a+b+2t}\ne0.$$
Thus $\alpha_t=0$, proving
$$m(u_a,u_b)=0\quad\text{if}\quad\dim\Span\{a,b,t\}=3.$$

\noindent\emph{Step 2: $\dim\Span\{a,b,t\}=2$ and $a,b\notin\C t$.}
First suppose $a+b\notin\C t$. Since $H$ spans $\C^n$ and $n\ge3$, choose $y\in H\setminus\Span\{a,b,t\}$ and put $r=b-y$. The vectors $y$ and $r-t=b-y-t$ are independent, so there exists $\eta$ with $\eta(r-t)=0$ and $\eta(y)\ne0$. Both triples $a,y,t$ and $a+b-y,y,t$ are independent. Applying equivariance to the zero value $m(u_a,u_y)$ gives
\begin{align*}
 0&=m(X_{r,\eta}u_a,u_y)+m(u_a,X_{r,\eta}u_y)\\
  &=\eta(a)m(u_{a+b-y},u_y)+\eta(y)m(u_a,u_b)\\
   &=\eta(y)m(u_a,u_b),
\end{align*}
where Step~1 eliminates the first term.

Now suppose $a+b\in\C t$. Choose $q\in H\setminus\Span\{a,b,t\}$ and $\eta$ with $\eta(q)=0$, $\eta(t)\ne0$. The operator $X_{t+q,\eta}$ is inner, and both terms on the right-hand side of
$$X_{t+q,\eta}m(u_a,u_b)=m(X_{t+q,\eta}u_a,u_b)+m(u_a,X_{t+q,\eta}u_b)$$
vanish by Step~1. Formula~\eqref{eq4.3} places $m(u_a,u_b)$ in the span of $u_z$ for $z\in K\setminus\{0,t\}$. For each such $z$,
$$X_{t+q,\eta}(u_z)=\eta(z)u_{z+t+q}.$$
Here $\eta(z)\ne0$, and the output indices are valid and pairwise distinct. Thus $m(u_a,u_b)=0$ also in this case.

\noindent\emph{Step 3: exactly one of $a,b$ belongs to $\C t$.}
By symmetry, write $b=\beta t$, where $\beta\in\C\setminus\{0,1\}$, and assume $a\notin\C t$. Choose $y\in H$ so that $a,t,y$ are independent, and put $r=b-y$. Because $\beta\ne1$, the vector $y$ is not in $\Span\{a,(\beta-1)t-y\}$. Choose $\eta$ satisfying
$$\eta((\beta-1)t-y)=\eta(a)=0,\qquad\eta(y)\ne0.$$
Then $X_{r,\eta}$ is inner, $X_{r,\eta}u_a=0$, and $X_{r,\eta}u_y=\eta(y)u_b$. Equivariance applied to $m(u_a,u_y)=0$ from Step~1 gives $0=\eta(y)m(u_a,u_b)$.

\noindent\emph{Step 4: $a,b\in\C t$.}
Write $a=\alpha t$ and $b=\beta t$, where $\alpha,\beta\in\C\setminus\{0,1\}$. Choose $y\in H\setminus\C t$, put $r=b-y$, and choose $\eta$ with
$$\eta((\beta-1)t-y)=0,\qquad\eta(y)\ne0.$$
Such a functional exists since $\beta\ne1$. Step~3 gives $m(u_a,u_y)=0$. The pair $(a+b-y,y)$ is covered by the second part of Step~2, since both indices are off $\C t$ and their sum
belongs to $\C t$. Hence
\begin{align*}
 0&=m(X_{r,\eta}u_a,u_y)+m(u_a,X_{r,\eta}u_y)\\
  &=\eta(a)m(u_{a+b-y},u_y)+\eta(y)m(u_a,u_b)\\
  &=\eta(y)m(u_a,u_b).
\end{align*}
These cases exhaust $\Lambda\times\Lambda$, proving $m=0$.
\end{proof}

For the $S$-construction, the Leibniz rule makes associative multiplication equivariant under inner derivations. Its nonvanishing when the bracket is nonzero provides the comparison with Theorem~\ref{thm4.2}.

\begin{proposition}\label{prop4.3}
If $S(A,\mathfrak g)$ is nonabelian, then $\cM(S(A,\mathfrak g))\ne0$.
\end{proposition}

\begin{proof}
Associative multiplication gives a symmetric bilinear map $m_A:S^2A\to A$, $m_A(f,g)=fg$. Expanding \eqref{eq2.2} in its last column and applying the Leibniz rule gives
$$[f_1,\ldots,f_{n-1},fg]_S=g[f_1,\ldots,f_{n-1},f]_S+f[f_1,\ldots,f_{n-1},g]_S.$$
Thus $m_A\in\cM(S(A,\mathfrak g))$. If $m_A=0$, then $A^2=0$, so every term in the defining determinant vanishes, contradicting nonabelianity. Hence $m_A\ne0$.
\end{proof}

\begin{corollary}\label{cor4.4}
The algebra $P(H,t)$ is not isomorphic to any nonabelian full algebra $S(A,\mathfrak g)$.
\end{corollary}

\begin{proof}
Apply Theorem~\ref{thm4.2}, Proposition~\ref{prop4.3}, and Lemma~\ref{lem4.1}.
\end{proof}

This comparison concerns the bracket on the underlying vector space $A$. Taking a derived subalgebra and a central quotient can change $\cM$; the central reduction constructed below provides an explicit instance.

\section{\texorpdfstring{$1/n$}{1/n}-derivations for $P(H,t)$ and $W(A, \mathfrak{g})$}\label{sec5}

The calculation of $\cQ(P(H,t))$ reduces to a functional equation for coefficient functions on $H$. We first record a density lemma that justifies the choices of subgroup elements used in the argument.

\begin{lemma}\label{lem5.1}
For every positive integer $m$, the set $H^m$ is Zariski dense in $(\C^n)^m$.  Consequently, finitely many nonzero polynomial conditions on $m$ free vectors can be satisfied simultaneously by vectors in $H$.
\end{lemma}

\begin{proof}
Choose a complex basis $b_1,\ldots,b_n$ from $H$.  Then $G=\Z b_1+\cdots+\Z b_n\subseteq H$.  A polynomial vanishing on $\Z^N$ is zero, by induction on $N$ and the fact that a nonzero one-variable polynomial has finitely many roots.  Applying this with $N=nm$ after the invertible linear change of coordinates determined by the $b_i$ proves density of $G^m$, and hence of $H^m$.  Apply this density statement to the product of the finitely many nonzero polynomials for the last assertion.
\end{proof}

We will also impose affine constraints that can be eliminated using
addition and subtraction in $H$, such as $a_1+\cdots+a_k=h$ with $h\in H$, by setting $a_k=h-a_1-\cdots-a_{k-1}$.  The remaining variables then range freely over $H$.  The conditions used below require determinants to be nonzero
or vectors to lie outside specified finite sets or proper
affine subspaces. Lemma~\ref{lem5.1} allows these conditions
to be satisfied simultaneously whenever each excluded set
is a proper algebraic subset of the parameter space. We call choices justified in this way \emph{generic choices in $H$}.

Fix $r\in H$.  An $n$-tuple $(a_1,\ldots,a_n)\in\Lambda^n$ is called \emph{$r$-admissible} if
$$a_i+r\in\Lambda\ (1\leq i\leq n), \quad \sum_{i=1}^n a_i+t\in\Lambda, \quad  \sum_{i=1}^n a_i+r+t\in\Lambda.$$

These admissibility conditions allow us to compare coefficients in the $1/n$-derivation identity away from the deleted indices. The following proposition determines the possible coefficient functions for each fixed shift.

\begin{proposition}\label{prop5.2}
Let $c:\Lambda\to\C$ satisfy $c(a)=0$ whenever $a+r\in\{0,t\}$.  Suppose that
\begin{align}\label{eq5.1}
 \sum_{i=1}^n c(a_i) \det(a_1,\ldots,a_i+r,\ldots,a_n)=n\det(a_1,\ldots,a_n) c(a_1+\cdots+a_n+t)
\end{align}
for every $r$-admissible tuple.  Then
$c=\left\{\begin{array}{lllll}
const, & r=0,\\
0, & r\neq 0.
\end{array}\right.$
\end{proposition}

\begin{proof}
We divide the argument according to the value of $r$.
\noindent\emph{The case $r\neq0$.}
Set $\Omega=\{a\in\Lambda:a\notin\C r,\ a+r\in\Lambda\}.$
We first prove
\begin{equation}\label{eq5.2}
 c(x+y)=c(x)+c(y) \quad\text{whenever }x,y,x+y\in\Omega.
\end{equation}

Call a pair $(p,q)$ regular if $r,p,q$ are linearly independent and the tuple
$$(p,q,p+q,z_4,\ldots,z_n)$$
is $r$-admissible for some $z_4,\ldots,z_n\in H$ such that $r,p,q,z_4,\ldots,z_n$ is a basis of $\C^n$.  When $n=3$, there are no $z_j$, so regularity includes admissibility of $(p,q,p+q)$.  For a regular pair put
$$\Delta=\det(p,q,r,z_4,\ldots,z_n).$$
The determinant in \eqref{eq5.1} vanishes before any column is shifted.  Among the shifted determinants, only the first three can be nonzero, and they are respectively $-\Delta,-\Delta,\Delta$. Thus
\begin{equation}\label{eq5.3}
 c(p+q)=c(p)+c(q)
\end{equation}
for every regular pair.

Now let $x,y,x+y\in\Omega$ be arbitrary.  A generic choice in $H$ gives $v\in H$ and, for $n\geq4$, the required completion vectors, such that each of
$$(x,v),\qquad (y,x+v),\qquad (x+y,v)$$
is regular.  Indeed, the failures of membership in $\Omega$, linear independence, and admissibility exclude only finitely many proper algebraic subsets of the ambient parameter space $\C^n$ for $v$, with $v$ restricted to the Zariski-dense subgroup $H$.  Once $v$ is fixed, the completion vectors for each pair range over $(\C^n)^{n-3}$, with the vectors restricted to the Zariski-dense set $H^{n-3}$. The corresponding basis determinant is a nonzero polynomial because $H$ spans $\C^n$, while every admissibility failure is a proper affine condition.  Lemma~\ref{lem5.1} therefore applies.  Applying \eqref{eq5.3} three times gives
\begin{align*}
 c(x+v)&=c(x)+c(v),\\
 c(x+y+v)&=c(y)+c(x+v),\\
 c(x+y+v)&=c(x+y)+c(v).
\end{align*}
Eliminating the terms involving $v$ proves \eqref{eq5.2}. This bridge argument also covers $n=3$ when the original triple $(x,y,x+y)$ is not admissible.

For $h\in H$, choose $a\in H$ generically so that $a,a+h\in\Omega$, and define
$$\ell(h)=c(a+h)-c(a).$$
Such an $a$ exists because the rejected values form a finite union of points and translates of the proper subspace $\C r$. This does not depend on $a$.  In fact, if $a$ and $b$ are two admissible choices, choose $v$ generically so that all uses of \eqref{eq5.2} below are valid, and put $w=b-a-v$.  Every required membership in $\Omega$ excludes a point or a translate of $\C r$ in the single free variable $v$, so Lemma~\ref{lem5.1} again gives such a choice.  Then
\begin{align*}
 c(a+h)-c(a)
 &=c(a+h+v)-c(a+v)\\
 &=c(a+h+v+w)-c(a+v+w)\\
 &=c(b+h)-c(b).
\end{align*}
Choosing one base point simultaneously for $h,k,h+k$ shows that $\ell(h+k)=\ell(h)+\ell(k)$.  Moreover, for $a\in\Omega$, choosing $b,a+b\in\Omega$ gives
\begin{equation}\label{eq5.4}
 c(a)=\ell(a).
\end{equation}

An additive map on an arbitrary subgroup $H$ need not admit a
complex-linear extension.  Instead, we use integer dilations to show
that $\ell=0$.  Fix $h\in H\setminus\C r$.  Choose
$b_1,\ldots,b_n\in H\setminus\C r$ such that
$$b_1+\cdots+b_n=h,\qquad B=(b_1,\ldots,b_n)\text{ is invertible}.$$
To justify this choice, take $b_1,\ldots,b_{n-1}$ freely and put $b_n=h-\sum_{i<n}b_i$.  The determinant becomes $\det(b_1,\ldots,b_{n-1},h)$, which is a nonzero polynomial since $h\neq0$.  Each condition $b_i\notin\C r$ also excludes a proper algebraic subset.  Lemma~\ref{lem5.1} therefore applies.

Write $q=B^{-1}r=(q_1,\ldots,q_n)^T$.  For all but finitely many positive integers $m$, every $m b_i$ belongs to $\Omega$, the tuple $(m b_1,\ldots,m b_n)$ is $r$-admissible, and $m h+t$ belongs to $\Omega$.  Indeed, the vectors $b_i$ and $h$ lie outside $\C r$; all remaining failures require one of finitely many affine equations in $m$, each having at most one solution.  Since
$$\det(m b_1,\ldots,m b_i+r,\ldots,m b_n)=m^n\left(1+\frac{q_i}{m}\right)\det B,$$
equation~\eqref{eq5.1}, together with \eqref{eq5.4}, gives
$$\sum_{i=1}^n\left(1+\frac{q_i}{m}\right)\ell(m b_i)=n\ell(mh+t).$$
The additivity of $\ell$ implies $\ell(mb_i)=m\ell(b_i)$ and
$\ell(mh+t)=m\ell(h)+\ell(t)$.  Hence
$$(n-1)m\ell(h)=\sum_{i=1}^nq_i\ell(b_i)-n\ell(t).$$
The right-hand side is independent of $m$.  Comparing two permissible distinct positive integers shows that $\ell(h)=0$.  Thus $\ell$ vanishes on $H\setminus\C r$.  If $h\in H\cap\C r$, choose $y\in H\setminus\C r$; then $h+y\notin\C r$, so $\ell(h)=\ell(h+y)-\ell(y)=0$.  It follows from \eqref{eq5.4} that $c=0$ on $\Omega$.

If $a\in\Lambda\setminus\C r$ but $a\notin\Omega$, then $a+r\in\{0,t\}$ and the boundary hypothesis gives $c(a)=0$.  It remains to consider $a\in\Lambda\cap\C r$.  Write $a=\kappa r$ with $\kappa\in\C\setminus\{0\}$.  If $a+r\in\{0,t\}$, the result is again part of the hypothesis.  Otherwise choose $b_2,\ldots,b_n\in\Lambda\setminus\C r$ generically so that $A=(a,b_2,\ldots,b_n)$ is invertible and $r$-admissible, and so that $A\one+t\notin\C r$.  The free parameter space is $H^{n-1}$: the determinant polynomial is nonzero because $a\neq0$ can be completed to a complex basis, and all remaining failures are proper affine conditions. All terms in \eqref{eq5.1} except the one containing $c(a)$ vanish. Moreover,
$$A^{-1}r=(\kappa^{-1},0,\ldots,0)^T,$$
so the remaining equation is $(1+\kappa^{-1})c(a)=0.$
The coefficient is nonzero because $a+r\neq0$.  Thus $c(a)=0$ and the case $r\neq0$ is complete.

\noindent\emph{The case $r=0$.}
Choose linearly independent $u_3,\ldots,u_n\in\Lambda$ and set
$$V=\Span\{u_3,\ldots,u_n\}, \qquad U=u_3+\cdots+u_n, \qquad  C_0=\sum_{j=3}^n c(u_j).$$
Call $(x,y)$ admissible here when $(x,y,u_3,\ldots,u_n)$ is $0$-admissible.  For every such pair with $\det(x,y,u_3,\ldots,u_n)\neq0$, equation \eqref{eq5.1} reduces to
\begin{equation}\label{eq5.5}
 c(x)+c(y)+C_0=n c(x+y+t+U).
\end{equation}

We claim that $c$ is constant on $\Lambda\setminus V$.  Let $x,z\in\Lambda\setminus V$.  Choose $y\in H$ generically so that \eqref{eq5.5} applies to all four pairs
$$(x,y),\quad (x+y+t+U,z),\quad (y,z),\quad (x,y+z+t+U).$$
The required conditions can be satisfied by choosing the image
of $y$ in the two-dimensional quotient $\C^n/V$ outside finitely
many affine lines. Using the first two pairs gives
\begin{align}\label{eq5.6}
 c(x)+c(y)+n c(z)+(n+1)C_0 =n^2c(x+y+z+2t+2U).
\end{align}
Using the last two gives
\begin{align}\label{eq5.7}
 n c(x)+c(y)+c(z)+(n+1)C_0=n^2c(x+y+z+2t+2U).
\end{align}
Comparison of \eqref{eq5.6} and \eqref{eq5.7} yields $(n-1)(c(x)-c(z))=0$.  Therefore there is $\gamma\in\C$ such that
\begin{equation}\label{eq5.8}
 c(a)=\gamma\qquad(a\in\Lambda\setminus V).
\end{equation}

Finally fix $a\in\Lambda$.  Parameterize the affine constraint $a_1+\cdots+a_n=a-t$ by freely choosing the first $n-1$ vectors in $H$.  Under this parameterization,
$$\det(a_1,\ldots,a_n) =\det(a_1,\ldots,a_{n-1},a-t),$$
which is a nonzero polynomial because $a-t\neq0$.  The conditions $a_i\in\Lambda\setminus V$ exclude further proper algebraic subsets. Lemma~\ref{lem5.1} therefore supplies such a tuple with nonzero determinant.  Its bracket index is $a$, so \eqref{eq5.1} and \eqref{eq5.8} give
$$n\gamma\det(a_1,\ldots,a_n) =n c(a)\det(a_1,\ldots,a_n).$$
Thus $c(a)=\gamma$.  Since $a$ was arbitrary, $c$ is constant.
\end{proof}

To pass from coefficient functions to endomorphisms, we decompose a $1/n$-derivation into components with fixed shifts. Proposition~\ref{prop5.2} then eliminates every nonzero shift and forces the remaining component to be scalar.

\begin{theorem}\label{thm5.3} 
We have $\cQ(P(H,t))=\C\id_{P(H,t)}.$
\end{theorem}

\begin{proof}
The putting
$$\deg u_a=a+\frac{t}{n-1}$$
makes $P(H,t)$ a graded algebra, with grading group $H+\Z\frac{t}{n-1}\subset\C^n$.  Let $T\in\cQ(P(H,t))$.  For $r\in H$ define its shift-$r$ component by
\begin{equation}\label{eq5.9}
 T_r(u_a)=c_r(a)u_{a+r},
\end{equation}
where $c_r(a)$ is the coefficient of $u_{a+r}$ in $T(u_a)$, and is set to zero when $a+r\notin\Lambda$.  Because $P(H,t)$ is an algebraic direct sum, for every $v\in P(H,t)$ the expansion $T(v)=\sum\limits_{r\in H}T_r(v)$ is finite, although the set of shifts occurring globally need not be finite or countable.

Comparing homogeneous components in the defining identity \eqref{eq4.2} shows that each $T_r$ separately lies in $\cQ(P(H,t))$.  For an $r$-admissible tuple, substitution of
\eqref{eq5.9} gives
$$\sum_{i=1}^n c_r(a_i)\det(a_1,\ldots,a_i+r,\ldots,a_n) = n\det(a_1,\ldots,a_n)c_r(a_1+\cdots+a_n+t).$$
The coefficient function $c_r$ therefore satisfies Proposition~\ref{prop5.2}, including its boundary condition. It follows that $T_r=0$ for $r\neq0$, while $T_0=\gamma\id_{P(H,t)}$ for some $\gamma\in\C$.  Since the decomposition is pointwise finite, this proves $T=\gamma\id_{P(H,t)}$.
\end{proof}

We next obtain a lower bound for $\dim_{\C}\cQ(W(A,\mathfrak g))$ from associative multiplication. This bound, together with Theorem~\ref{thm5.3}, yields the comparison with the $W$-construction.

\begin{proposition}\label{prop5.4}
If $W(A,\mathfrak g)$ is nonabelian, the assignment $a\mapsto L_a$, where $L_a(f)=af$, defines a linear embedding $A\hookrightarrow\cQ(W(A,\mathfrak g))$. Hence
$$\dim_{\C}\cQ(W(A,\mathfrak g))\geq\dim_{\C}A\geq2.$$
\end{proposition}

\begin{proof}
For $f,a\in A$, write
$$u(f)=(f,D_1f,\ldots,D_{n-1}f)^T,\qquad v(a)=(0,D_1a,\ldots,D_{n-1}a)^T.$$
The Leibniz rule gives $u(af)=a u(f)+f v(a)$.  Let $C=(u(f_1),\ldots,u(f_n))$, and denote by $C_i(v(a))$ the matrix obtained by replacing column $i$ with $v(a)$.  Multilinearity of the determinant gives
\begin{align}\label{eq5.10}
 \sum_{i=1}^n[f_1,\ldots,af_i,\ldots,f_n]_W=na\det(C)+\sum_{i=1}^n f_i\det C_i(v(a)).
\end{align}
If $\mathcal C_{ki}$ is the $(k,i)$-cofactor of $C$, with rows numbered from $0$, then
$$\sum_{i=1}^n f_i\det C_i(v(a))=\sum_{k=1}^{n-1}D_k(a)\sum_{i=1}^n C_{0i}\mathcal C_{ki}=0.$$
The last equality is the cofactor identity obtained by replacing row $k$ of $C$ with row $0$.  Equation \eqref{eq5.10} is exactly the condition $L_a\in\cQ(W(A,\mathfrak g))$.

It remains to prove injectivity.  The annihilator
$$\Ann(A)=\{a\in A:aA=0\}$$
is a $\mathfrak g$-invariant ideal: if $aA=0$, then $D(a)b=D(ab)-aD(b)=0$ for all $D\in\mathfrak g$ and $b\in A$.  The $\mathfrak g$-simplicity of $A$ implies that $\Ann(A)$ is either $0$ or $A$. The latter case gives $A^2=0$ and makes the bracket \eqref{eq2.3} identically zero, contrary to nonabelianity. Hence $\Ann(A)=0$, so $L_a=0$ implies $a=0$.

Finally, $\dim A\neq1$: otherwise $\dim\Der(A)\leq\dim\End_{\C}(A)=1$, whereas the prescribed subalgebra $\mathfrak g\subseteq\Der(A)$ has dimension $n-1\geq2$.
\end{proof}

\begin{corollary}\label{cor5.5}
The algebra $P(H,t)$ is not isomorphic to any simple nonabelian $W(A,\mathfrak g)$.
\end{corollary}

\begin{proof}
Use Theorem~\ref{thm5.3}, Proposition~\ref{prop5.4}, and the functoriality in Lemma~\ref{lem4.1}.
\end{proof}

\section{Inner derivations of the SW-construction}\label{sec6}

Let $L=SW(A,D)$ and give it the componentwise $A$-module structure. Set
$$\mathcal K_A=\Inder(L)\cap\End_A(L).$$
By Theorem~\ref{thm3.2}, $\Inder P(H,t)$ is simple. The comparison with the $SW$-construction therefore reduces to proving that $\mathcal K_A$ is a nonzero proper ideal of $\Inder(L)$.

For $c\in A$, define the first-order operator
$$\Delta_c(f^{\langle i\rangle})=(cD(f))^{\langle i\rangle},\qquad 1\leq i\leq n-1.$$
It satisfies
\begin{equation}\label{eq6.1}
 \Delta_c(av)=cD(a)v+a\Delta_c(v) \qquad(a\in A,\ v\in L).
\end{equation}

\begin{lemma}\label{lem6.1}
Every $\delta\in\Inder(L)$ admits a decomposition
\begin{equation}\label{eq6.2}
 \delta=\Delta_c+\phi \qquad\text{with }c\in A\text{ and }\phi\in\End_A(L).
\end{equation}
\end{lemma}

\begin{proof}
It is enough to consider an elementary inner derivation $\delta=\operatorname{ad}(v_1,\ldots,v_{n-1})$ with homogeneous fixed arguments. Let $s$ be the number of colors occurring among them. If $s\leq n-3$, then every bracket obtained by inserting a variable misses a color, so $\delta=0$.  Only the following two cases can yield a nonzero operator.

If $s=n-1$, one element of every color occurs. After a permutation of the fixed arguments, write $v_i=a_i^{\langle i\rangle}$ and put $c_0=a_1\cdots a_{n-1}$.  Formula \eqref{eq2.4} gives, up to one common sign $\varepsilon\in\{1,-1\}$,
\begin{align}\label{eq6.3}
 \delta(f^{\langle i\rangle})=\varepsilon\left(\left(\prod_{j\neq i}a_j\right)D(a_i)f-c_0D(f)\right)^{\langle i\rangle}.
\end{align}
The first term in parentheses is $A$-linear in $f$, while the second is $-\varepsilon\Delta_{c_0}$.  Hence \eqref{eq6.2} holds.

If $s=n-2$, one color $p$ is missing and one color $q$ occurs twice. The operator is zero on every $A^{\langle i\rangle}$ with $i\neq p$.  On the remaining component it has the form
\begin{equation}\label{eq6.4}
 f^{\langle p\rangle}\longmapsto\pm(f\theta)^{\langle q\rangle}
\end{equation}
for a fixed $\theta\in A$, and is therefore $A$-linear.  Thus \eqref{eq6.2} holds with $c=0$.  Finite linear combinations of elementary inner derivations retain a decomposition of the stated form.
\end{proof}

The preceding decomposition allows us to prove that $\mathcal K_A$ is stable under commutators with all inner derivations. Nonabelianity and $D$-simplicity will then ensure that this ideal is nonzero and proper.

\begin{proposition}\label{prop6.2}
Assume that $A$ is $D$-simple and that $SW(A,D)$ is nonabelian. Then $\mathcal K_A$ is a nonzero proper ideal of $\Inder(SW(A,D))$.
\end{proposition}

\begin{proof}
First we prove ideality.  Let $\eta\in\mathcal K_A$ and $\delta=\Delta_c+\phi\in\Inder L$ be a decomposition from Lemma~\ref{lem6.1}.  Both $\phi$ and $\eta$ are $A$-linear, so $[\phi,\eta]$ is $A$-linear.  Moreover, by \eqref{eq6.1},
\begin{align*}
 [\Delta_c,\eta](av)
 &=\Delta_c(a\eta(v))-\eta(cD(a)v+a\Delta_c(v))\\[1mm]
 &=a[\Delta_c,\eta](v).
\end{align*}
Thus $[\delta,\eta]$ is $A$-linear.  It is also inner because $\Inder L$ is a Lie algebra and both $\delta$ and $\eta$ are inner. Therefore
$$[\Inder(L),\mathcal K_A]\subseteq\mathcal K_A.$$

To see that $\mathcal K_A\neq0$, choose a nonzero bracket of homogeneous elements.  Exactly one color, say $q$, occurs twice and every other color occurs once.  Choose a color $p\neq q$ and regard the element of color $p$ as the variable.  The other $n-1$ elements define an inner derivation of the type \eqref{eq6.4}; it is $A$-linear and nonzero because its value on the omitted element is the chosen nonzero bracket.

It remains to prove properness.  The annihilator $\Ann(A)$ is $D$-invariant, by the same Leibniz-rule argument used in Proposition~\ref{prop5.4}.  Since $A$ is $D$-simple, it is either $0$ or $A$.  If it were $A$, then $A^2=0$ and \eqref{eq2.4} would vanish identically.  Hence
\begin{equation}\label{eq6.5}
\Ann(A)=0.
\end{equation}

Write the chosen nonzero homogeneous bracket, after reordering, as
$$0\neq p\bigl(D(a)b-aD(b)\bigr),$$
where $a,b$ have the repeated color and $p$ is the product of the coefficients in all other colors.  At least one of $pbD(a)$ and $paD(b)$ is nonzero.  In the first case set $c=pb$ and $f=a$; in the second set $c=pa$ and $f=b$.  In either case
\begin{equation}\label{eq6.6}
 cD(f)\neq0.
\end{equation}
Choose one fixed element in each color, taking the coefficient $b$ or $a$ in the repeated color according to the preceding choice and the coefficients whose product is $p$ in the other colors. The resulting inner derivation $\delta$ is of the full-color type \eqref{eq6.3}, and its first-order part is $\pm\Delta_c$.  By \eqref{eq6.5} and \eqref{eq6.6}, there exists $g\in A$ with $cD(f)g\neq0$. For every color $i$, the $A$-linear part of $\delta$ cancels in the difference
$$\delta((fg)^{\langle i\rangle})-f\delta(g^{\langle i\rangle})=\pm(cD(f)g)^{\langle i\rangle}\neq0.$$
Thus $\delta$ is not $A$-linear, so $\mathcal K_A\neq\Inder L$.
\end{proof}

\begin{corollary}\label{cor6.3}
The algebra $P(H,t)$ is not isomorphic to any simple nonabelian $SW(A,D)$.
\end{corollary}

\begin{proof}
An $n$-Lie algebra isomorphism conjugates the corresponding inner derivation Lie algebras. The Lie algebra $\Inder P(H,t)$ is simple by Theorem~\ref{thm3.2}, whereas $\Inder SW(A,D)$ has the nonzero proper ideal from Proposition~\ref{prop6.2}.
\end{proof}

\begin{proof}[Proof of Theorem~\ref{thm1.1}]
The distinguishing invariants are $\cM(L)$, $\cQ(L)$, and $\Inder(L)$ for the $S$, $W$, and $SW$ constructions, respectively. The assertions follow from Corollaries~\ref{cor4.4}, \ref{cor5.5}, and~\ref{cor6.3}.
\end{proof}

\section{A central reduction of the \texorpdfstring{$S$}{S} construction}\label{sec7}

The comparison with the three literal families does not exclude a
realization obtained by taking a derived subalgebra and then a central
quotient.  In fact, every algebra $P(H,t)$ considered here has such a
realization in the $S$ construction.

\begin{proposition}\label{prop7.1}
There exists an $n$-dimensional abelian Lie subalgebra $\mathfrak g_t\subseteq\Der(\C[H])$ such that $\C[H]$ is $\mathfrak g_t$-simple and
$$P(H,t)\cong \frac{[S, \dots, S]}{\langle\mathbbm{1}\rangle},\qquad S=S(\C[H],\mathfrak g_t),$$
where $\langle\mathbbm{1}\rangle$ is a central ideal of $[S, \dots, S]$.
\end{proposition}

\begin{proof}
Write $A_H=\C[H]=\bigoplus\limits_{a\in H}\C x^a$, where $x^a x^b=x^{a+b}$.  Choose a basis $\lambda_1,\ldots,\lambda_n$ of $(\C^n)^*$ such that
$$\lambda_1(t)=1,\qquad \lambda_i(t)=0\quad(2\leq i\leq n),\qquad \det\bigl(\lambda_i(a_j)\bigr)_{i,j=1}^n=\det(a_1,\ldots,a_n).$$
The last normalization is obtained by rescaling one of $\lambda_2,\ldots,\lambda_n$.  Define Euler derivations $\partial_i(x^a)=\lambda_i(a)x^a$ and set
$$D_1=x^t\partial_1,\qquad D_i=\partial_i\quad(2\leq i\leq n),\qquad\mathfrak g_t=\Span\{D_1,\ldots,D_n\}.$$
These derivations commute, because the $\partial_i$ commute and $\partial_i(x^t)=0$ for $i\geq2$.  They are $A_H$-linearly independent: after replacing $D_1$ by $x^{-t}D_1$, a relation evaluated on every $x^a$ gives a linear relation among the coordinate functionals $\lambda_i$ on the spanning subgroup $H$.  In particular, $\dim\mathfrak g_t=n$.

Let $I\ne0$ be a $\mathfrak g_t$-invariant associative ideal of $A_H$. Since $x^{-t}D_1=\partial_1$, it is invariant under every $\partial_i$. Choose $f=\sum\limits_{a\in F}c_a x^a\in I\setminus\{0\}$ with the number of nonzero coefficients minimal.  If $F$ contains distinct $a,b$, then $\lambda_i(a)\ne\lambda_i(b)$ for some $i$, and $(\partial_i-\lambda_i(a)\id)f$ is a nonzero element of $I$ with smaller support.  Hence $f$ is a nonzero scalar multiple of a monomial. Every monomial is a unit, so $I=A_H$.  Thus $A_H$ is $\mathfrak g_t$-simple.

The bracket of $S(A_H,\mathfrak g_t)$, taken in the displayed basis of derivations, is
$$[x^{a_1},\ldots,x^{a_n}]=\det(a_1,\ldots,a_n)x^{a_1+\cdots+a_n+t}.$$
Consequently its derived subalgebra is
\begin{equation}\label{eq7.1}
S(A_H,\mathfrak g_t)^{(1)}=\bigoplus_{b\in H\setminus\{t\}}\C x^b.
\end{equation}
Indeed, a bracket with output index $t$ has $a_1+\cdots+a_n=0$, and its determinant vanishes.  Conversely, if $b\ne t$, choose $q_1,\ldots,q_{n-1}\in H$ such that $q_1,\ldots,q_{n-1},b-t$ are linearly independent.  The bracket with input indices
$$q_1,\ldots,q_{n-1},\quad b-t-q_1-\cdots-q_{n-1}$$
is a nonzero multiple of $x^b$.  This proves \eqref{eq7.1}.

Since $t\ne0$, the derived subalgebra contains $1=x^0$.
The determinant bracket makes $\C 1$ central.  Under $x^a\mapsto e_a$,
\eqref{eq7.1} is precisely
$\widetilde{\cA}(H,t)$, and quotienting by $\C 1$ gives $P(H,t)$.
\end{proof}

\section{The realization of \texorpdfstring{$E(H)$}{E(H)} as \texorpdfstring{$W(A,\mathfrak g)$}{W(A,g)}}

We now realize Pozhidaev's second construction directly as an algebra of the $W$ family.  As in the rest of this paper, the ground field in this section is $\C$, so it has characteristic zero.

Let $H\subseteq\C^n$ be an additive subgroup containing $t_1=(1-n,0,\ldots,0)$, and set
$$H_1=\{h=(h_1,\ldots,h_n)\in H:h_1=1\}.$$
Assume that $H_1$ contains $n$ linearly independent vectors over $\C$. Pozhidaev's construction \cite[Section~3, equation~(8)]{PozhidaevCentral} defines
$$E(H)=\bigoplus_{a\in H_1}\C e_a,\qquad[e_{a_1},\ldots,e_{a_n}]=\det(a_1,\ldots,a_n)e_{a_1+\cdots+a_n+t_1}.$$
The output index belongs to $H_1$, since its first coordinate is $n+(1-n)=1$.  By \cite[Theorem~3.2]{PozhidaevCentral}, $E(H)$ is a central simple $n$-Lie algebra.

Choosing a base point in $H_1$ identifies the underlying vector space of $E(H)$ with a group algebra. We then absorb the shift in the transported bracket into a commuting family of derivations.

\begin{proposition}
The algebra $E(H)$ is isomorphic to an $n$-Lie algebra of $W$-type.
\end{proposition}

\begin{proof}
Choose $b=(1,\beta)\in H_1$ and put
$$\Gamma=\{\gamma\in\C^{n-1}:(0,\gamma)\in H\},\qquad s=(n-1)\beta.$$
Then
$H_1=\{(1,\beta+\gamma):\gamma\in\Gamma\}$, and $s\in\Gamma$ because $(n-1)b+t_1=(0,s)\in H$.  Differences of $n$ linearly independent elements of $H_1$ give $n-1$ linearly independent vectors in $\{0\}\times\Gamma$. Hence $\Span\Gamma=\C^{n-1}$.

Let $A=\C[\Gamma]$, with $x^\gamma x^\eta=x^{\gamma+\eta}$, and define $\delta_i\in\Der(A)$ by $\delta_i(x^\gamma)=\gamma_i x^\gamma$. The linear bijection $\Phi:E(H)\to A$ given by
$$\Phi(e_{(1,\beta+\gamma)})=x^\gamma$$
transports the bracket to
\begin{equation}\label{eq8.1}
[f_1,\ldots,f_n]=x^s\det\begin{pmatrix}
f_1&\cdots&f_n\\
\delta_1(f_1)&\cdots&\delta_1(f_n)\\
\vdots&&\vdots\\
\delta_{n-1}(f_1)&\cdots&\delta_{n-1}(f_n)
\end{pmatrix}.
\end{equation}
Indeed, subtracting $\beta_i$ times the first row from row $i+1$ removes $\beta$ from the coefficient determinant, while the output index is $(1,\beta+\gamma_1+\cdots+\gamma_n+s)$.

If $s=0$, set $D_i=\delta_i$.  If $s\ne0$, choose $k$ with
$s_k\ne0$ and set
$$D_k=x^s\delta_k,\qquad D_i=\delta_i-\frac{s_i}{s_k}\delta_k\quad(i\ne k).$$
These derivations commute. The $\delta_i$ are $A$-linearly independent, since evaluation of a relation $\sum\limits_i f_i\delta_i=0$ on the units $x^\gamma$ yields $\sum\limits_i\gamma_i f_i=0$ for every $\gamma\in\Gamma$, and $\Gamma$ spans $\C^{n-1}$.  The change from the $\delta_i$ to the $D_i$ is invertible over $A$.  Consequently $\mathfrak g=\Span\{D_1,\ldots,D_{n-1}\}$ is abelian and has dimension $n-1$.

Every $\mathfrak g$-invariant associative ideal $I$ of $A$ is invariant under the $\delta_i$, since each $\delta_i$ is an $A$-linear combination of the $D_j$.  If $I\ne0$, choose $f=\sum\limits_{\gamma\in F}c_\gamma x^\gamma\in I\setminus\{0\}$ with the number of nonzero coefficients minimal.  If $F$ contains distinct $\gamma,\eta$, some $i$ satisfies $\gamma_i\ne\eta_i$, and $(\delta_i-\gamma_i\id)f$ is a nonzero element of $I$ with smaller support.  Therefore $F$ consists of one element.  It follows that $f$ is a unit, so $I=A$.  This proves $\mathfrak g$-simplicity.

For $s\ne0$, subtracting the prescribed multiples of the $\delta_k$ row from the other derivation rows and multiplying that row by $x^s$ multiplies the determinant by $x^s$.  Thus \eqref{eq8.1} is exactly the determinant defining $W(A,\mathfrak g)$ in the ordered basis $D_1,\ldots,D_{n-1}$. The same conclusion is immediate for $s=0$, and $\Phi$ is the desired isomorphism.
\end{proof}

\begin{example}
For $H=\Z^n$, take $\beta=0$.  Then $\Gamma=\Z^{n-1}$ and
$$A=\C[x_1^{\pm1},\ldots,x_{n-1}^{\pm1}],\qquad\mathfrak g=\bigoplus_{i=1}^{n-1}\C x_i\frac{\partial}{\partial x_i},$$
with $\Phi(e_{(1,a_1,\ldots,a_{n-1})})=x_1^{a_1}\cdots x_{n-1}^{a_{n-1}}$.
\end{example}

\begin{remark}
The two realizations in the last two sections describe the relation with the determinant constructions precisely: $E(H)$ is a $W$-algebra, while $P(H,t)$ is a central quotient of the derived algebra of a full $S$-algebra.
\end{remark}

These realizations lead to the following existence question, with passage to derived subalgebras and central quotients explicitly allowed.

\begin{question}
For $n\geq3$, does there exist a simple infinite-dimensional complex $n$-Lie algebra that cannot be obtained from the $S$, $W$, and $SW$ constructions above by taking finitely many derived subalgebras and central quotients?
\end{question}

\medskip\noindent\begin{minipage}{\textwidth}
\textbf{Funding.} The second author was supported by the Basic Research Program of Jiangsu (BK20251784).
 
%\smallskip\noindent\textbf{Data availability.} No datasets were generated or analysed during this study.

\smallskip\noindent\textbf{Conflict of interest.} On behalf of all authors, the corresponding author states that there is no conflict of interest.
\end{minipage}\par

\subsection*{Declaration of generative AI assistance}
During preparation of this manuscript, OpenAI was used to assist with language revision and LaTeX organization. The authors remain fully responsible for the manuscript's mathematical content.

\end{document}